\documentclass{amsart}

\usepackage{amssymb}
\usepackage{graphicx}
\usepackage{array,url,xcolor,ulem}

\newtheorem{theorem}{Theorem}[section]
\newtheorem{lemma}[theorem]{Lemma}

\theoremstyle{definition}

\theoremstyle{remark}

\numberwithin{equation}{section}

\begin{document}

\title[L--space knots without Khovanov thin surgery]{An infinite family of strongly invertible L--space knots without Khovanov thin surgery}

\author{Marc Kegel}
\address{Universidad de Sevilla, Dpto.\ de Álgebra,
Avda.\ Reina Mercedes s/n,
41012 Sevilla, Spain}
\email{kegelmarc87@gmail.com}

\author{Masakazu Teragaito}
\address{Department of Mathematics Education, Hiroshima University,
1-1-1 Kagamiyama, Higashi-hiroshima, Japan 739-8524.}
\email{teragai@hiroshima-u.ac.jp}


\date{\today}

\begin{abstract}
Currently, there exist only three strongly invertible L--space knots that are known to admit no Khovanov thin surgery.
In this article, we give the first infinite family of such knots.
\end{abstract}

\maketitle

\section{Introduction}

In \cite{Wa}, Watson introduced an invariant $\varkappa(K,\Phi)$ for a knot $K$ with strong inversion $\Phi$,
which is a finite dimensional $\mathbb{Z}$-graded vector space.
In particular, $\varkappa(K,\Phi)\cong0$ if and only if $K$ is the unknot.

Watson \cite{Wa0} proposed a conjecture that
a strongly invertible knot  $K$ with a strong inversion $\Phi$ is an L--space knot if and only if 
the Watson invariant $\varkappa(K,\Phi)$ is supported in a single diagonal grading $\delta=q-2h$.
For one of the implications of this conjecture,
Baker, Kegel, and McCoy \cite{BKM}
gave two counterexamples.
These are the knots $t09847$ and $o9\_30634$ in the SnapPy census; for details on the census, we refer to~\cite{D1,D2,L,ABG,BKM2,BKM3}.
Interestingly, both share another exceptional property. Their formal semigroups \cite{BCG,W,Te} are closed under addition.
This observation leads to one more counterexample, $o10\_143807$ \cite{Te2}.
At present, these three knots are the only strongly invertible L--space knots whose Watson invariant is known to be thick.

In addition, these knots are shown to admit no Khovanov thin surgery.
This means that filling along any non-trivial slope yields a $3$-manifold which can never be the double branched cover of $S^3$ over
a Khovanov thin knot or link.

The purpose of this article is to give the first infinite family of hyperbolic L--space knots having
a (unique) strong inversion whose associated Watson invariant is thick and admitting no Khovanov thin surgery.

For an integer $n\ge 1$,
let $K_n$ be the closure of the $4$-braid
\[
[(2,1,3,2)^{2n+1},-1,2,1,1,2],
\]
where an integer $\pm i$ denotes the standard generator $\sigma_i^{\pm1}$ of the braid group $B_4$ of $4$-strands. For a surgery description of $K_n$, we refer to Figure~\ref{fig:knot}.
This knot $K_n$ is known to be a strongly invertible hyperbolic $L$--space knot whose formal semigroup
is closed under addition \cite{BK}.
In particular, $K_1$ is $o9\_30634$ in the SnapPy census.

\begin{figure}[htbp]
\includegraphics*[width=0.5\textwidth]{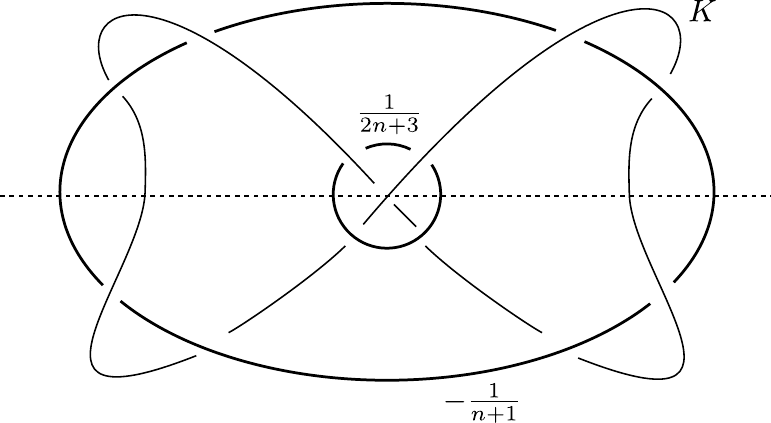}
\caption{A strongly invertible surgery description of $K_n$.
By performing the two surgeries shown here, $K$ is mapped to the knot $K_n$. A surgery coefficient $r$ of $K$ is mapped under this diffeomorphism to the surgery coefficient $8n+4+r$ of $K_n$.
The strong inversion $\Phi$ of $L$ given by the $\pi$-rotation along the dotted line induces a strong inversion of $K_n$.
}\label{fig:knot}
\end{figure}

\begin{theorem}\label{thm:main}
For each $n\geq1$, the knot $K_n$ as defined above is a strongly invertible L-space knot with a unique strong inversion $\Phi$ and its Watson invariant $\varkappa(K_n,\Phi)$ is thick.
More precisely, it is supported in at least two $\delta$-gradings.
\end{theorem}

In fact, computer calculations suggest that $\varkappa(K_n,\Phi)$ has width exactly two. From the braid words, it is readily seen that each $K_n$ is strongly quasi-positive, from which we can read off that the Seifert genus distinguishes all knots in the family. Moreover, we can show that the knots $K_n$ have no Khovanov thin surgery.

\begin{theorem}\label{thm:no_thin}
For each $n\geq1$, the knot $K_n$ has no Khovanov thin surgery.
\end{theorem}

Note that Theorem~\ref{thm:main} and Theorem~\ref{thm:no_thin} give further evidence for an affirmative answer to Question 1.4 from~\cite{BKM}.

\subsection*{Individual grant support}
The first author is supported by a Ram\'on y Cajal grant (RYC2023-043251-I) and the project PID2024-157173NB-I00 funded by MCIN/AEI/10.13039/ 50110001\-1033, by ESF+, and by FEDER, EU; and by a VII Plan Propio de Investigación y Transferencia (SOL2025-36103) of the University of Sevilla.
The second author has been supported by
JSPS KAKENHI Grant Number JP25K07004.

\section{Strong inversions}

We start by proving the first part of Theorem~\ref{thm:main}.

\begin{lemma}
    For each $n\geq1$, the knot $K_n$ as defined above is a hyperbolic, strongly invertible L-space knot with a unique strong inversion $\Phi$.
\end{lemma}

\begin{proof}
    In \cite{BK} it was shown that the knots $K_n$ are L-space knots, by identifying an exceptional surgery to an L-space.
    
    From the surgery description along the link $L$ in Figure~\ref{fig:knot}, it follows directly that each $K_n$ is strongly invertible. Using SnapPy~\cite{SnapPy}, we show that this is the only symmetry of $K_n$. 

    First, we establish that $L$ (which is called $L12n1739$ in the HTW census~\cite{HTW}) is hyperbolic and possesses a unique strong inversion. For that, we use SnapPy~\cite{SnapPy} to compute the symmetry group of orientation-preserving diffeomorphisms for the link $L$, and find that there exists exactly one non-trivial symmetry that extends to the link itself, which has to be the symmetry visible in Figure~\ref{fig:knot}. This unique strong inversion inherently survives the Dehn filling that yields the knot $K_n$. 

    By Thurston's hyperbolic Dehn filling theorem~\cite{Thurston}, we know that any filling of $L$ with all slopes sufficiently large is again hyperbolic and has the same symmetry group as $L$, and recent work of Futer--Purcell--Schleimer~\cite{FPS} gives a computable bound for that, cf.~\cite{BKM2}. We compute this bound and check that for $n\geq9$ the slopes are sufficiently large, so that we conclude that $K_n$ has a unique strong inversion if $n\geq9$. For the finitely many small values of $n$, we build the knot $K_n$ in SnapPy, verify that it is hyperbolic and that its symmetry group also has a unique strong inversion.

For the code and computations, we refer to~\cite{data}.
\end{proof}

\section{Branching set}

Figure \ref{fig:tangle}(a) shows the tangle $T_n$ obtained from the surgery description of $K_n$ shown in Figure \ref{fig:knot}
by taking the quotient under the strong inversion $\Phi$.
The outside of the circle gives the tangle, so its double branched cover restores the exterior of $K_n$.
The two vertical boxes labeled with an integer $m$ contain vertical $m$ half twists.
If $m>0$, then it is right-handed, otherwise left-handed.
Similarly, the horizontal box labeled with the integer $-(8n+6)$ contains $8n+6$ left-handed horizontal half twists.
For a rational tangle $r$, $T_n(r)$ denotes the knot or link obtained by filling the $r$-rational tangle to $T_n$.
In Figure \ref{fig:tangle}(a), the dotted lines indicate $0$-framing.
That is, it corresponds to the longitudinal surgery on $K_n$ in the double cover \cite{Mo}.
Figure \ref{fig:tangle}(b) shows $T_n(8n+6)$.

\begin{figure}[htbp]
\includegraphics*[width=0.9\textwidth]{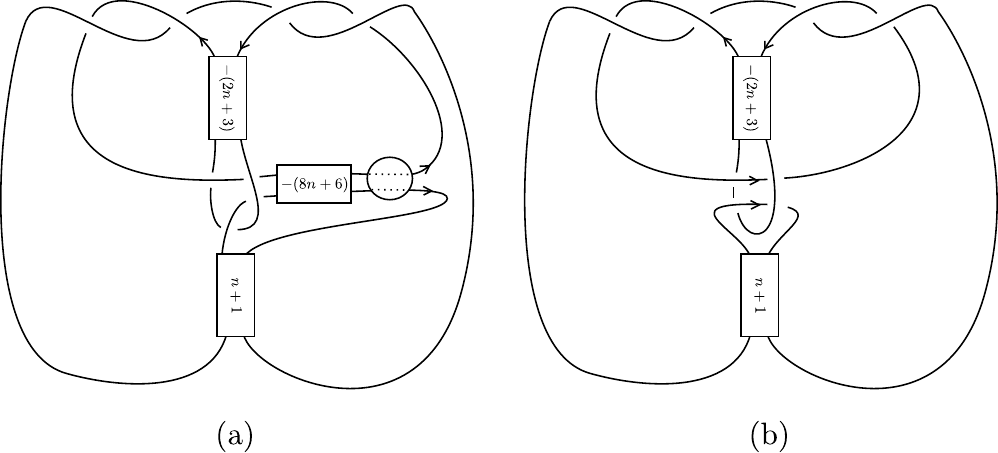}
\caption{(a) The outside of the circle is the tangle $T_n$, whose double branched cover is the exterior of $K_n$.
(b) shows the link $T_n(8n+6)$ which is obtained as the tangle filling of $T_n$ with the rational tangle $8n+6$.
}\label{fig:tangle}
\end{figure}

In this article, we use reduced Khovanov homology with coefficients in $\mathbb{Z}_2$ as done in~\cite{Wa0,Wa,BKM,Te2}.
For a knot or link $L$, we write
$\mathrm{Kh}_q^h(L)$ for the reduced Khovanov homology with $\mathbb Z_2$-coefficients in homological grading $h$ and quantum grading $q$.

\begin{lemma}\label{lem:N=8n+6}
For all $n\geq 1$, we have
\begin{align*}
&\mathrm{Kh}_{-9}^{-4}(T_n(8n+6))\cong\mathrm{Kh}_{4n+5}^{2n+4}(T_n(8n+6))\cong\mathbb{Z}_2,\\
&\mathrm{Kh}^{-3}_{-10}(T_n(8n+5))\cong\mathrm{Kh}^{2n+3}_{6n+10}(T_n(10n+11))\cong0.
\end{align*}
\end{lemma}

\begin{proof}
This follows from Lemmas \ref{lem:8n+6}, \ref{lem:8n+5} and \ref{lem:10n+11} below.
\end{proof}

Assuming Lemma~\ref{lem:N=8n+6}, we prove the following.

\begin{lemma}\label{lem:two-lines}
For every integer $N$, we have
\begin{align}
    \mathrm{Kh}_{N-8n-15}^{-4}(T_n(N))&\cong
\begin{cases}
\mathbb{Z}_2&  \text{if $N\ge 8n+5$},\\
\mathbb{Z}_2^2 & \text{if $N\le 8n+4$}.
\end{cases}\\
\mathrm{Kh}_{N-4n-1}^{2n+4}(T_n(N))&\cong
\begin{cases}
\mathbb{Z}_2^2&  \text{if $N\ge 10n+12$},\\
\mathbb{Z}_2 & \text{if $N\le 10n+11$}.
\end{cases}
\end{align}
\end{lemma}

\begin{proof}
First, we prove (3.1).
Suppose $N\ge 8n+7$.
Then the $N$-tangle and the box with the integer $-(8n+6)$ are merged into an integer tangle corresponding
to $N-(8n+6)\ge 1$.
Let us call this diagram $D$.
Then $n_-(D)=n+7$, where $n_-$ denotes the number of negative crossings.
Choose the leftmost crossing there, which is positive.
Let $D_-$ and $D_\infty$ be the resulting diagrams after resolving the crossing horizontally or vertically.
Then $D_-$ corresponds to $T_n(N-1)$, and $D_\infty$ represents the unknot.
In addition, $n_-(D_\infty)=n+5+(N-8n-7)$.
Thus $c=n_-(D_\infty)-n_-(D)=N-8n-9$, so
Viro's long exact sequence for a positive crossing (see for example \cite{Wa0}) yields
\begin{equation}\label{eq:long-positive}
\dots \to \mathrm{Kh}_{q-3c-2}^{h-c-1}(D_\infty) \to \mathrm{Kh}_q^h(D) \to \mathrm{Kh}_{q-1}^h(D_-) \to 
\mathrm{Kh}_{q-3c-2}^{h-c}(D_\infty) \to \cdots.
\end{equation}

For us, $(h,q)=(-4, N-8n-15)$.
Then $\mathrm{Kh}_{q-3c-2}^{h-c-1}(D_\infty)\cong0$, since
 $h-c-1=0$ and $q-3c-2=0$ cannot hold simultaneously.
Also, $\mathrm{Kh}_{q-3c-2}^{h-c}(D_\infty)\ncong 0$ only when  $N=8n+5$.
Thus $\mathrm{Kh}_q^h(D)\cong \mathrm{Kh}_{q-1}^h(D_-)$ under the assumption $N\ge 8n+7$.
By Lemma \ref{lem:N=8n+6}, we have $\mathrm{Kh}_{-9}^{-4}(T_n(8n+6))\cong\mathbb{Z}_2$.
Hence
\[
\dots\cong \mathrm{Kh}^{-4}_{-7}(T_n(8n+8)) \cong \mathrm{Kh}^{-4}_{-8}(T_n(8n+7))\cong \mathrm{Kh}^{-4}_{-9}(T_n(8n+6))\cong\mathbb{Z}_2.
\]

Suppose $N\le 8n+5$.
In this case, the merged tangle of the $N$-tangle and the box containing $8n+6$ left half twists is
an integer tangle corresponding to $N-(8n+6)\le -1$, which contains only negative crossings.
Then $n_-(D)=9n+13-N$.
As above, we choose the leftmost crossing there. Since $n_-(D_\infty)=n+5$,
$c=n_-(D_\infty)-n_-(D)=N-8n-8$.
Hence, the long exact sequence for a negative crossing yields
\begin{equation}\label{eq:long-negative}
\begin{split}
\dots &\to \mathrm{Kh}_{q-3c-1}^{h-c-1}(D_\infty) \to \mathrm{Kh}_{q+1}^h(D_-) \to \mathrm{Kh}_{q}^h(D) \\
&\to \mathrm{Kh}_{q-3c-1}^{h-c}(D_\infty) \to \mathrm{Kh}_{q+1}^{h+1}(D_-)\to \cdots.
\end{split}
\end{equation}
Recall $(h,q)=(-4,N-8n-15)$.
Then $\mathrm{Kh}_{q-3c-1}^{h-c-1}(D_\infty)\cong 0$, and
$\mathrm{Kh}^{h-c}_{q-3c-1}(D_\infty)\ncong 0$ only when $N=8n+4$.
This implies that
\[
\mathrm{Kh}^{-4}_{-9}(T_n(8n+6))\cong \mathrm{Kh}^{-4}_{-10}(T_n(8n+5)), 
\]
and
\[
\mathrm{Kh}^{-4}_{-11}(T_n(8n+4))\cong  \mathrm{Kh}^{-4}_{-12}(T_n(8n+3))\cong \mathrm{Kh}^{-4}_{-13}(T_n(8n+2))\cong \cdots.
\]
Finally, by using Lemma \ref{lem:N=8n+6},
\[
\begin{split}
0&\cong\mathrm{Kh}^{-1}_{0}(D_\infty) \to \mathrm{Kh}^{-4}_{-10}(T_n(8n+5))\to \mathrm{Kh}^{-4}_{-11}(T_n(8n+4))\\
&\to \mathrm{Kh}^{0}_{0}(D_\infty) \to \mathrm{Kh}^{-3}_{-10}(T_n(8n+5))\cong0.
\end{split}
\]
Thus
\[
\mathrm{Kh}^{-4}_{-11}(T_n(8n+4))\cong\mathbb{Z}_2^2,
\]
which implies (3.1).

The argument for (3.2) is similar to the argument for (3.1).
Set $(h,q)=(2n+4,N-4n-1)$. Suppose $N\ge 8n+7$.
The long exact sequence (\ref{eq:long-positive}) yields
\begin{equation}
\begin{split}
\dots &\to \mathrm{Kh}_{20n+24-2N}^{10n+12-N}(D_\infty) \to \mathrm{Kh}_{N-4n-1}^{2n+4}(D) \to \mathrm{Kh}_{N-4n-2}^{2n+4}(D_-)\\
& \to 
\mathrm{Kh}_{20n+24-2N}^{10n+13-N}(D_\infty) \to \cdots.
\end{split}
\end{equation}
Since $\mathrm{Kh}_{20n+24-2N}^{10n+12-N}(D_\infty)\cong\mathrm{Kh}_{20n+24-2N}^{10n+13-N}(D_\infty)\cong0$, if $N\ne 10n+12$, we have
\[
\cdots \cong\mathrm{Kh}_{6n+13}^{2n+4}(T_n(10n+14))\cong \mathrm{Kh}_{6n+12}^{2n+4}(T_n(10n+13))\cong\mathrm{Kh}_{6n+11}^{2n+4}(T_n(10n+12)),
\]
and
\[
\mathrm{Kh}^{2n+4}_{6n+10}(T_n(10n+11))\cong\mathrm{Kh}^{2n+4}_{6n+9}(T_n(10n+10)) \cong \cdots \cong \mathrm{Kh}^{2n+4}_{4n+5}(T_n(8n+6))\cong\mathbb{Z}_2.
\]
When $N=10n+12$, (\ref{eq:long-positive}) reduces to
\[
0\cong\mathrm{Kh}^{2n+3}_{6n+10}(D_-)\to \mathrm{Kh}^0_{0}(D_\infty) \to \mathrm{Kh}^{2n+4}_{6n+11}(D) \to \mathrm{Kh}^{2n+4}_{6n+10}(D_-) \to \mathrm{Kh}^1_0 (D_\infty) \cong0
\]
by Lemma \ref{lem:N=8n+6}.
Thus $\mathrm{Kh}_{6n+11}^{2n+4}(T_n(10n+12))\cong\mathbb{Z}_2^2$.

Finally, assume $N\le 8n+5$.
As in (3.1), the long exact sequence (\ref{eq:long-negative}) yields
\begin{equation*}
\mathrm{Kh}_{20n+22-2N}^{10n+11-N}(D_\infty) \to \mathrm{Kh}_{N-4n}^{2n+4}(D_-) \to \mathrm{Kh}_{N-4n-1}^{2n+4}(D) \\
\to \mathrm{Kh}_{20n+22-2N}^{10n+12-N}(D_\infty).
\end{equation*}
Since $\mathrm{Kh}_{20n+22-2N}^{10n+11-N}(D_\infty)\cong\mathrm{Kh}_{20n+22-2N}^{10n+12-N}(D_\infty)\cong0$,
we have
\[
\mathrm{Kh}^{2n+4}_{4n+5}(T_n(8n+6))\cong\mathrm{Kh}^{2n+4}_{4n+4}(T_n(8n+5))\cong \mathrm{Kh}^{2n+4}_{4n+3}(T_n(8n+4))\cong \cdots.
\]
This completes the proof.
\end{proof}

Before proving Theorem~\ref{thm:main}, we briefly recall the definition of the Watson invariant \cite{Wa}.
For an integer $i$, 
let $A_i$ be $\oplus_{h,q} \mathrm{Kh}^h_q(T_n(i))$.
Then the long exact sequence  yields
a homomorphism $f_i\colon A_i \to A_{i-1}$.
Let $\mathbb{A}$ be the inverse limit of the system $(A_i,f_i)_{i\in \mathbb{Z}}$.
This means  that $\mathbb{A}$ consists of
sequences $(x_i)_{i\in \mathbb{Z}}$ such that $x_i \in A_i$ and $f_i(x_i)=x_{i-1}$ for any $i$.
Let $\mathbb{K}$ be the subspace of $\mathbb{A}$ consisting of sequences such that 
$x_i=0$ for all $i$ sufficiently small.
Then the Watson invariant $\varkappa(K_n,\Phi)$ is the quotient space $\mathbb{A}/\mathbb{K}$.

Since the linear map $f_i$ preserves the  homological grading, there is an absolute $\mathbb{Z}$-grading $h$ on $\varkappa$.
Thus $\varkappa(K_n,\Phi)=\oplus_{h} \varkappa^h(K_n,\Phi)$.
Also, $\varkappa$ inherits a relative $q$-grading, so there exists a relative $\delta$-grading, with $\delta=q-2h$.

\begin{proof}[Proof of Theorem \ref{thm:main}]
By Lemma \ref{lem:two-lines},
$\varkappa^h(K_n,\Phi)\ncong 0$ for $h=-4$ and $2n+4$.
In fact, $\varkappa^{-4}$ has rank $\geq 1$, and it has a generator having
$\delta$-grading $q-2h=-8n-7$, where $q$ and $h$ are normalized at $N=0$.

Similarly,
$\varkappa^{2n+4}$ has rank $\ge 1$ having a generator with $\delta=-8n-9$.
Thus $\varkappa(K_n,\Phi)$ is supported in at least two $\delta$-gradings.
\end{proof}

\section{The link $T_n(8n+6)$}

In this section, we prove the first half of Lemma \ref{lem:N=8n+6}.
That is,  we calculate $\mathrm{Kh}^{-4}_{-9}(T_n(8n+6))$ and 
$\mathrm{Kh}^{2n+4}_{4n+5}(T_n(8n+6))$ for the two-component link $T_n(8n+6)$
as illustrated in Figure \ref{fig:tangle}(b).

For convenience, we introduce an oriented diagram $D^i_j$ for integers $i<0$ and $j>0$ as shown in Figure \ref{fig:D-ij}(a).
A vertical box with integer $i$ contains vertical $|i|$ left-handed half twists.
Another box with integer $j$ contains vertical $j$ right-handed half twists.
Thus $T_n(8n+6)$ is represented by $D^{-(2n+3)}_{n+1}$.

\begin{figure}[htbp]
\includegraphics*[width=0.8\textwidth]{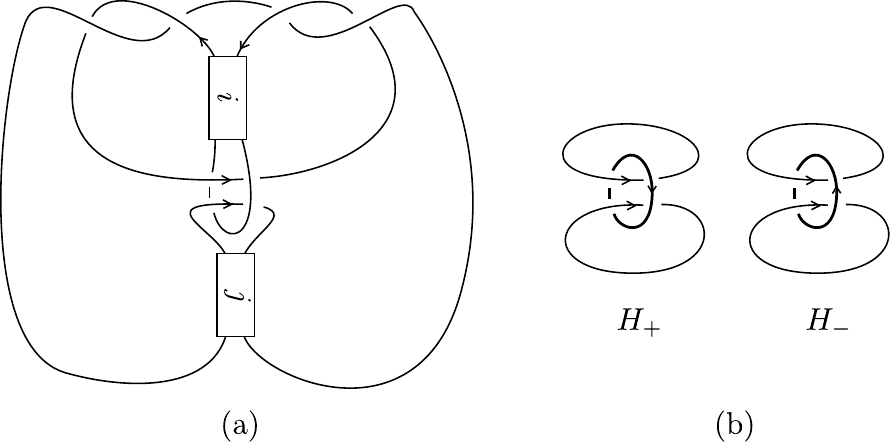}
\caption{(a) The oriented diagram $D^i_j$.
(b) After resolving the topmost positive crossing in the box with $i$ horizontally,  we have the connected sum of two Hopf links.
It is $H_+$ when $i$ is even, and $H_-$ when $i$ is odd.
}\label{fig:D-ij}
\end{figure}

Choose the topmost positive crossing in the box with $i$.
We resolve this crossing in two ways, as usual.
Then the horizontal resolution yields the connected sum of two Hopf links $H_\pm$ as shown in Figure \ref{fig:D-ij}(b).
If $i$ is even, it is $H_+$, otherwise $H_-$.
The vertical resolution gives $D^{i-1}_{j}$, whose orientation is not induced from that of $D^i_j$, but
we preferentially keep the (local) orientation indicated by arrows in Figure \ref{fig:D-ij}(a) and extend it to the whole.

\begin{lemma}\label{lem:c}
The diagram $D^i_j$ has
\[
n_-(D^i_j)=
\begin{cases}
j+6 & \text{if $i$ is odd},\\
j+2 & \text{if $i$ is even}.
\end{cases}
\]
\end{lemma}

\begin{proof}
This easily follows from a direct calculation in the diagram $D^i_j$.
\end{proof}

\begin{lemma}\label{lem:H2}
The Khovanov homology groups of the links $H_\pm$ are
\[
\mathrm{Kh}^h_q(H_\pm)\cong
\begin{cases}
\mathbb{Z}_2 & \text{if $(h,q)=(\pm 4,\pm 10), (0, \pm 2)$}\\
\mathbb{Z}_2^2 &  \text{if $(h,q)=(\pm 2, \pm 6)$}\\
0 & \text{otherwise.}
\end{cases}
\]
\end{lemma}

\begin{proof}
We perform this calculation using KnotJob \cite{S}. 
\end{proof}

\begin{figure}[htbp]
\includegraphics*[width=0.8\textwidth]{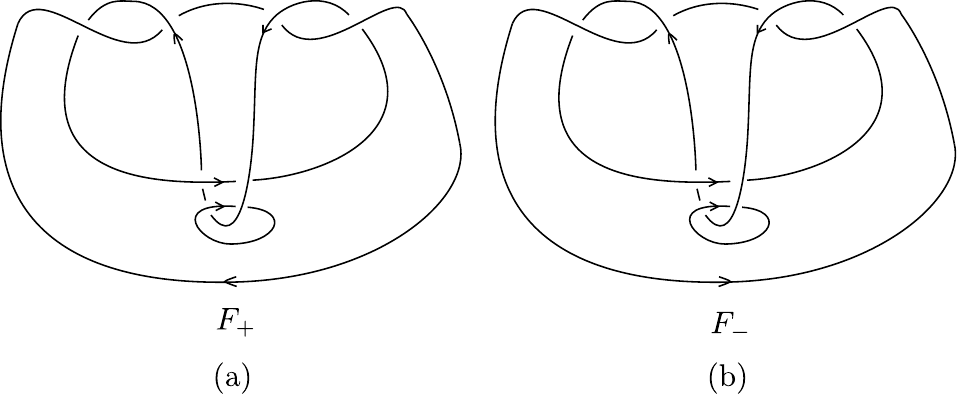}
\caption{
The diagrams $F_+$ and $F_-$.
}\label{fig:F}
\end{figure}

\begin{lemma}\label{lem:F}
The links $F_\pm$ shown in Figure \ref{fig:F} have Khovanov homology
\[
\mathrm{Kh}^h_q(F_\pm)\cong
\begin{cases}
\mathbb{Z}_2 & \text{if $(h,q)=(-2,-2), (-1,0), (5,12)$} \\
\mathbb{Z}_2^2 & \text{if $(h,q)=(1,4), (3,8), (4,10)$} \\
\mathbb{Z}_2^3 & \text{if $(h,q)=(0,2)$}\\
\mathbb{Z}_2^4 & \text{if $(h,q)=(2,6)$}\\
0 & \text{otherwise}.
\end{cases}
\]
\end{lemma}

\begin{proof}
We perform the calculation with KnotJob~\cite{S}.
\end{proof}

\begin{lemma}\label{lem:8n+6}
$
\mathrm{Kh}^{-4}_{-9}(D^{-(2n+3)}_{n+1})\cong \mathrm{Kh}^{2n+4}_{4n+5}(D^{-(2n+3)}_{n+1})\cong\mathbb{Z}_2.
$
\end{lemma}

\begin{proof}
By resolving the topmost crossing in the upper box of $D^{-(2n+3)}_{n+1}$,
we have the long exact sequence
\begin{equation}\label{eq:long}
\begin{split}
\cdots &\to \mathrm{Kh}^{h-1}_{q-1}(H_-)\to \mathrm{Kh}^{h-c-1}_{q-3c-2}(D^{-(2n+2)}_{n+1}) \to
\mathrm{Kh}^h_q(D^{-(2n+3)}_{n+1})\\
& \to
\mathrm{Kh}^{h}_{q-1} (H_-) \to\mathrm{Kh}^{h-c}_{q-3c-2}(D^{-(2n+2)}_{n+1}) \to
\cdots.
\end{split}
\end{equation}
First, set $(h,q)=(-4,-9)$.
By Lemma \ref{lem:H2}, $\mathrm{Kh}^{-5}_{-10}(H_-)\cong0$ and $\mathrm{Kh}^{-4}_{-10}(H_-)\cong\mathbb{Z}_2$.
Since $c=-4$ by Lemma \ref{lem:c}, we have
\[
\mathrm{Kh}^{-1}_{1}(D^{-(2n+2)}_{n+1}) \to
\mathrm{Kh}^{-4}_{-9}(D^{-(2n+3)}_{n+1}) \to \mathbb{Z}_2 \to \mathrm{Kh}_{1}^{0}(D^{-(2n+2)}_{n+1}).
\]
Hence, it suffices to show that $\mathrm{Kh}^{-1}_{1}(D^{-(2n+2)}_{n+1})\cong\mathrm{Kh}_{1}^{0}(D^{-(2n+2)}_{n+1})\cong0$.
For the diagram $D^{-(2n+2)}_{n+1}$, we apply the resolution at the topmost crossing as above.
Note that $c=4$ by Lemma \ref{lem:c}.
Then the long exact sequence reduces to
\[
\mathrm{Kh}^{-6}_{-13}(D^{-(2n+1)}_{n+1})
\to\mathrm{Kh}^{-1}_{1}(D^{-(2n+2)}_{n+1})\to \mathrm{Kh}^{-1}_{0}(H_+)\cong0,
\]
and\[
\mathrm{Kh}^{-5}_{-13}(D^{-(2n+1)}_{n+1})
\to\mathrm{Kh}^{0}_{1}(D^{-(2n+2)}_{n+1})\to \mathrm{Kh}^{0}_{0}(H_+)\cong0.
\]
Thus, we need to know that
$\mathrm{Kh}^{-6}_{-13}(D^{-(2n+1)}_{n+1})\cong\mathrm{Kh}^{-5}_{-13}(D^{-(2n+1)}_{n+1})\cong0$.
By resolving the topmost crossing in the upper box of $D^{-(2n+1)}_{n+1}$, we obtain
\[
0\cong\mathrm{Kh} ^{h-1}_{-14}(H_-)\to  \mathrm{Kh}^{h+3}_{-3}(D^{-2n}_{n+1})\to \mathrm{Kh}^{h}_{-13}(D^{-(2n+1)}_{n+1}) \to
\mathrm{Kh}^h_{-14}(H_-)\cong0,
\]
where $h\in \{-6,-5\}$.  Thus  $\mathrm{Kh}^{h+3}_{-3}(D^{-2n}_{n+1})\cong \mathrm{Kh}^{h}_{-13}(D^{-(2n+1)}_{n+1})$.
Then
$\mathrm{Kh}^{-3}_{-3}(D^{-2n}_{n+1})\cong\mathrm{Kh}^{-2}_{-3}(D^{-2n}_{n+1})\cong0$ implies 
$\mathrm{Kh}^{-1}_{1}(D^{-(2n+2)}_{n+1})\cong\mathrm{Kh}_{1}^{0}(D^{-(2n+2)}_{n+1})\cong0$.
Inductively,
it is enough to see that 
\[
\mathrm{Kh}^{-2n-3}_{-4n-3}(D^0_{n+1})\cong\mathrm{Kh}^{-2n-2}_{-4n-3}(D^0_{n+1})\cong0.
\]
For $D^0_{n+1}$, we resolve the topmost negative crossing in the lower box.
Then
\[
\mathrm{Kh}_{-4n-2}^h(F_\pm) \to \mathrm{Kh}^h_{-4n-3}(D^0_{n+1}) \to \mathrm{Kh}^{h+1}_{-4n-1}(D^0_n) \to
\mathrm{Kh}_{-4n-2}^{h+1}(F_\pm)
\] 
for $h \in \{-2n-3,-2n-2\}$.
The sign of $F_\pm$ depends on the parity of $n$.
If $n$ is even (resp.\ odd), then 
the resulting diagram is $F_+$ (resp.\ $F_-$).
By Lemma \ref{lem:F},
$\mathrm{Kh}_{-4n-2}^h(F_\pm) \cong\mathrm{Kh}_{-4n-2}^{h+1}(F_\pm)\cong0$.
Hence, 
$\mathrm{Kh}^h_{-4n-3}(D^0_{n+1}) \cong \mathrm{Kh}^{h+1}_{-4n-1}(D^0_n)$ for $h\in
\{-2n-3,-2n-2\}$.
Inductively, it is enough to see that
\[
\mathrm{Kh}^{-n-2}_{-2n-1}(D^0_0)\cong\mathrm{Kh}^{-n-1}_{-2n-1}(D^0_0)\cong0.
\]
However, $D_0^0$ is the Hopf link with linking number one.
It has
$\mathrm{Kh}^h_q(D^0_0)\cong\mathbb{Z}_2$
 for $(h,q)=(0,1), (2,5)$, and $0$ otherwise.
This finishes the computation of $\mathrm{Kh}^{-4}_{-9}(D^{-(2n+3)}_{n+1})$

Next, we set $(h,q)=(2n+4,4n+5)$.
Then (\ref{eq:long}) reduces to
\[
\mathrm{Kh}^{2n+3}_{4n+4}(H_-)\to \mathrm{Kh}^{2n+7}_{4n+15}(D^{-(2n+2)}_{n+1}) \to
\mathrm{Kh}^{2n+4}_{4n+5}(D^{-(2n+3)}_{n+1})\to
\mathrm{Kh}^{2n+4}_{4n+4} (H_-).
\]
By Lemma \ref{lem:H2}, $\mathrm{Kh}^{2n+3}_{4n+4}(H_-)\cong\mathrm{Kh}^{2n+4}_{4n+4}(H_-)\cong0$.
Thus 
\[
\mathrm{Kh}^{2n+7}_{4n+15}(D^{-(2n+2)}_{n+1}) \cong
\mathrm{Kh}^{2n+4}_{4n+5}(D^{-(2n+3)}_{n+1}).
\]
For $D_{n+1}^{-(2n+2)}$, resolve the topmost crossing in the upper box as before.
Then
\[
\mathrm{Kh}^{2n+6}_{4n+14}(H_+)\to \mathrm{Kh}^{2n+2}_{4n+1}(D^{-(2n+1)}_{n+1})\to
\mathrm{Kh}^{2n+7}_{4n+15}(D^{-(2n+2)}_{n+1}) \to \mathrm{Kh}^{2n+7}_{4n+14}(H_+).
\]
By Lemma \ref{lem:H2} again,
$\mathrm{Kh}^{2n+6}_{4n+14}(H_+)\cong
\mathrm{Kh}^{2n+7}_{4n+14}(H_+)\cong0$.
So, we have 
\[
\mathrm{Kh}^{2n+2}_{4n+1}(D^{-(2n+1)}_{n+1})\cong
\mathrm{Kh}^{2n+7}_{4n+15}(D^{-(2n+2)}_{n+1}).
\]
Inductively, we need to know that $\mathrm{Kh}^{2}_1(D^{-1}_{n+1})\cong\mathbb{Z}_2$.
Resolving the remaining crossing in the upper twist, we have
\begin{equation}\label{eq:D1n+1}
0\cong\mathrm{Kh}^{1}_{0}(H_-) \to \mathrm{Kh}^{5}_{11}(D^0_{n+1})
\to \mathrm{Kh}^2_1(D^{-1}_{n+1}) \to \mathrm{Kh}^2_{0}(H_-)\cong0.
\end{equation}
For $D^0_{n+1}$, we resolve the topmost crossing in the lower box as above.
Then
\begin{equation}\label{eq:n}
\mathrm{Kh}^{5}_{13}(D_n^0)\to
\underset{\cong\mathbb{Z}_2}{\mathrm{Kh}_{12}^{5}(F_\pm)} \to \mathrm{Kh}^{5}_{11}(D^0_{n+1}) \to \mathrm{Kh}^{6}_{13}(D^0_n) \to
\mathrm{Kh}_{12}^{6}(F_\pm)\cong0.
\end{equation}
We calculate $\mathrm{Kh}^h_{13}(D^0_{n})$ for $h\in \{5,6\}$.
By resolving the topmost crossing in the lower box,
\begin{equation}\label{eq:D0n}
0\cong\mathrm{Kh}^h_{14}(F_\pm)\to \mathrm{Kh}^h_{13}(D^0_{n}) \to \mathrm{Kh}^{h+1}_{15}(D^0_{n-1}) \to
\mathrm{Kh}^{h+1}_{14}(F_\pm)\cong0.
\end{equation}
Thus $\mathrm{Kh}^h_{13}(D^0_n) \cong \mathrm{Kh}^{h+1}_{15}(D^0_{n-1})$.
Inductively, we have 
\[
\mathrm{Kh}^{h}_{13}(D^0_n)\cong \mathrm{Kh}^{h+n}_{13+2n}(D^0_0)
\]
for $h\in \{5,6\}$.
As mentioned above, $\mathrm{Kh}^{h+n}_{13+2n}(D^0_0)\cong0$.
Thus (\ref{eq:n}) reduces to
\[
0 \to \mathbb{Z}_2 \to \mathrm{Kh}^5_{11}(D^0_{n+1}) \to 0,
\]
so  $\mathrm{Kh}^5_{11}(D^0_{n+1})\cong\mathbb{Z}_2$ as desired.
\end{proof}

\section{The knot $T_n(8n+5)$}

\begin{lemma}\label{lem:8n+5}
$\mathrm{Kh}^{-3}_{-10}(T_n(8n+5))\cong0$.
\end{lemma}

\begin{proof}
Resolve the negative crossing marked with $\ast$ in Figure \ref{fig:8n+5}(a).
Then the long exact sequence gives
\[
0\cong\mathrm{Kh}^{-1}_{-2}(O) \to \mathrm{Kh}^{-3}_{-9}(T_n(8n+6)) \to \mathrm{Kh}^{-3}_{-10}(T_n(8n+5)) \to
\mathrm{Kh}^{0}_{-2}(O) \cong0
\]
where $O$ denotes the unknot.
\begin{figure}[htbp]
\includegraphics*[width=0.9\textwidth]{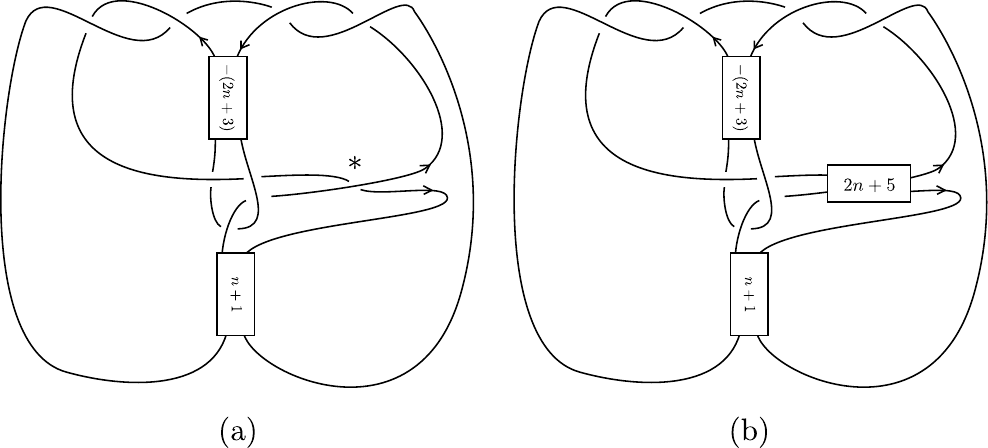}
\caption{(a) The knot $T_n(8n+5)$.   
Resolving the crossing marked with $\ast$ yields
the unknot and $T_n(8n+6)$.  (b) The knot $T_n(10n+11)$.
The box labeled with $2n+5$ contains $2n+5$ horizontal right-handed half twists.
}\label{fig:8n+5}
\end{figure}
Hence we need to know that $\mathrm{Kh}^{-3}_{-9}(T_n(8n+6))\cong0$.
With the notation from the previous section, this is $\mathrm{Kh}^{-3}_{-9}(D^{-(2n+3)}_{n+1})$.
In (\ref{eq:long}), set $(h,q)=(-3,-9)$.
By Lemma \ref{lem:H2}, we have
\[
 \mathrm{Kh}^0_{1}(D_{n+1}^{-(2n+2)})\to \mathrm{Kh}^{-3}_{-9}(D^{-(2n+3)}_{n+1})
\to 0.
\]
In the proof of Lemma \ref{lem:8n+6},
we showed that $\mathrm{Kh}^0_{1}(D^{-(2n+2)}_{n+1})\cong0$.
Thus  we have $\mathrm{Kh}^{-3}_{-9}(D^{-(2n+3)}_{n+1})\cong0$.
\end{proof}

\section{The knot $T_n(10n+11)$}

\begin{lemma}\label{lem:10n+11}
$\mathrm{Kh}^{2n+3}_{6n+10}(T_n(10n+11))\cong0$.
\end{lemma}

\begin{proof}
Resolve the leftmost crossing in the box labeled with $2n+5$.
Since $c=2n+2$, we have the long exact sequence
\[
0\cong \mathrm{Kh}^{0}_{2}(O)\to \mathrm{Kh}^{2n+3}_{6n+10}(T_n(10n+11)) \to \mathrm{Kh}^{2n+3}_{6n+9}(T_n(10n+10)) \to
\mathrm{Kh}^{1}_{2}(O)\cong0,
\]
where $O$ denotes the unknot again.
Thus $\mathrm{Kh}^{2n+3}_{6n+10}(T_n(10n+11)) \cong \mathrm{Kh}^{2n+3}_{6n+9}(T_n(10n+10))$.
Inductively,
it is enough to show that $\mathrm{Kh}^{2n+3}_{4n+5} (T_n(8n+6))\cong0$.
Recall that $T_n(8n+6)$ is represented by the diagram $D^{-(2n+3)}_{n+1}$.

Set $(h,q)=(2n+3,4n+5)$ for (\ref{eq:long}).  Then
\[
0\cong\mathrm{Kh}^{2n+2}_{4n+4}(H_-)\to \mathrm{Kh}^{2n+6}_{4n+15}(D^{-(2n+2)}_{n+1}) \to
\mathrm{Kh}^{2n+3}_{4n+5}(D^{-(2n+3)}_{n+1})\to
\mathrm{Kh}^{2n+3}_{4n+4} (H_-)\cong0.
\]
Thus
$\mathrm{Kh}^{2n+6}_{4n+15}(D^{-(2n+2)}_{n+1}) \cong
\mathrm{Kh}^{2n+3}_{4n+5}(D^{-(2n+3)}_{n+1})$.
For $D_{n+1}^{-(2n+2)}$,
we have
\[
0\cong\mathrm{Kh}^{2n+5}_{4n+14}(H_+)\to \mathrm{Kh}^{2n+1}_{4n+1}(D^{-(2n+1)}_{n+1})\to
\mathrm{Kh}^{2n+6}_{4n+15}(D^{-(2n+2)}_{n+1}) \to \mathrm{Kh}^{2n+6}_{4n+14}(H_+)\cong0
\]
as in the proof of Lemma \ref{lem:8n+6}.
Hence 
$\mathrm{Kh}^{2n+1}_{4n+1}(D^{-(2n+1)}_{n+1})\cong
\mathrm{Kh}^{2n+6}_{4n+15}(D^{-(2n+2)}_{n+1})$.
Inductively,
it is enough to show that
$\mathrm{Kh}^{-1}_{1}(D^{-1}_{n+1})\cong0$.
As in (\ref{eq:D1n+1}),
\[
0\cong\mathrm{Kh}^{-2}_{0}(H_-) \to \mathrm{Kh}^{2}_{11}(D^0_{n+1})
\to \mathrm{Kh}^{-1}_1(D^{-1}_{n+1}) \to \mathrm{Kh}^{-1}_{0}(H_-)\cong0.
\]
For $D^0_{n+1}$,
as in (\ref{eq:n}), we have
\[
0\cong\mathrm{Kh}_{12}^{2}(F_\pm) \to \mathrm{Kh}^{2}_{11}(D^0_{n+1}) \to \mathrm{Kh}^{3}_{13}(D^0_n) \to
\mathrm{Kh}_{12}^{3}(F_\pm)\cong0.
\]
Thus, it suffices to show that $\mathrm{Kh}^3_{13}(D^0_{n})\cong0$.
As in (\ref{eq:D0n}),
\[
0\cong\mathrm{Kh}^{3}_{14}(F_\pm)\to \mathrm{Kh}^{3}_{13}(D^0_{n}) \to \mathrm{Kh}^{4}_{15}(D^0_{n-1}) \to
\mathrm{Kh}^{4}_{14}(F_\pm)\cong0.
\]
Hence 
$\mathrm{Kh}^{3}_{13}(D^0_{n}) \cong \mathrm{Kh}^{4}_{15}(D^0_{n-1})$.
Inductively, this is isomorphic to $\mathrm{Kh}^{n+3}_{2n+13}(D^0_0)$.
Recall that $D^0_0$ is the Hopf link with linking number one.
It has
$\mathrm{Kh}^h_q(D^0_0)\cong\mathbb{Z}_2$
 for $(h,q)=(0,1), (2,5)$, and $0$ otherwise.
Hence $\mathrm{Kh}^{n+3}_{2n+13}(D^0_0)\cong0$.
\end{proof}

\section{Thin slopes}

In the remainder of this article, we show that for each $n\geq1$, the knot $K_n$ has no Khovanov thin surgery, thereby proving Theorem~\ref{thm:no_thin}. Since $K_n$ is hyperbolic, Thurston's hyperbolic Dehn surgery theorem~\cite{Thurston} implies that, for all sufficiently long slopes $r$, the filled manifold $K_n(r)$ is hyperbolic and its symmetry group injects into the symmetry group of $K_n$ (see for example~\cite[Corollary~2.5]{BKM2}). As the symmetry group of $K_n$ is isomorphic to $\mathbb{Z}_2$ and contains a unique strong inversion, the Montesinos trick~\cite{Mo} implies that for all sufficiently long slopes $r$, the manifold $K_n(r)$ admits a unique description as a double branched cover of $S^3$, namely the one with branching set $T_n(r)$ constructed in the previous sections.

In the previous sections, we have shown that for every integral slope $r\in\mathbb Z$ the branching set $T_n(r)$ is thick. Lemma~\ref{lem:rational_thick} below is a slight adaptation of Proposition~5.2 in \cite{W} and proves that $T_n(r)$ is thick for any rational slope. Consequently, a slope $r$ can be a thin slope for $K$ only if it is either an exceptional slope (that is, $K_n(r)$ is not hyperbolic) or a symmetry-exceptional slope (that is, $K_n(r)$ is hyperbolic but has a symmetry group strictly larger than that of $K_n$). We rule out both cases in the next two subsections.

\begin{lemma}\label{lem:rational_thick}
Let $K$ be a strongly invertible knot with tangle exterior $T$ such that for any slope $r$ the $r$-surgery $K(r)$ is diffeomorphic to the double branched cover of the tangle filling $T(r)$. If for an integer $n\in\mathbb Z$ the tangle fillings $T(n)$ and $T(n+1)$ are thick, then the tangle fillings $T(r)$ are thick for all $r\in\mathbb Q\cap[n,n+1]$.
\end{lemma}

\begin{proof}
The proof follows the argument of Watson
\cite[Proposition~5.2]{W}. The induction on the continued fraction
expansion of $r$ uses only the skein long exact sequence and the existence of
non-trivial Khovanov homology in at least two $\delta$-gradings for the integer
closures $T(n)$ and $T(n+1)$. Consequently, the same argument
shows that every rational closure $T(r)$ has Khovanov homology supported in at
least two $\delta$-gradings, and hence is Khovanov thick.
\end{proof}

\subsection{Exceptional slopes}

We first classify the exceptional slopes of the $K_n$.

\begin{lemma}\label{lem:exc}
    For each $n\geq 1$, the only exceptional slope of $K_n$ $($other than $\infty$$)$ is $8n+6$, yielding the Seifert fibered space $M \left(-1; \dfrac{1}{2} , \dfrac{2n+1}{4n+4} , \dfrac{2}{4n+5} \right)$.
\end{lemma}

\begin{proof}
In~\cite{BKM} it was already shown that $14$ is the only exceptional slope of $K_1$, yielding the claimed Seifert fibered space. Thus, we can assume that $n>1$.
We consider again the surgery description 
\[
        K_n=L\left(\frac{1}{2n+3},*,-\frac{1}{n+1}\right)
\]
from Figure~\ref{fig:knot}, where the $*$ denotes that the second cusp of $L$ is left unfilled. Any exceptional filling $r$ of $K_n$ corresponds to the exceptional filling 
$$\left(\frac{1}{2n+3},r-(8n+4),-\frac{1}{n+1}\right)$$ 
of $L$. The 6-theorem of Agol~\cite{Agol} and Lackenby~\cite{Lackenby} says that if all slopes of $L$ are longer than $6$, the filled manifold is again hyperbolic and thus the filling is not exceptional. 

    Using SnapPy~\cite{SnapPy}, we compute the short slopes on the first and third cusps. The slopes $1/(2n+3)$ and $-1/(n+1)$ are not short for every $n\geq 2$. Thus, any exceptional filling of $K_n$ must arise from a short slope on the second cusp of $L$. The set of short slopes on this cusp is $\{\infty,-1,0,1,2,3,4,5\}$.
    The slope $\infty$ corresponds to the meridional filling and yields $S^3$. 
    And it was already shown in Lemma~2.2 of \cite{BK} that the slope $(2,1)$ of $L$, which gets mapped to the slope $8n+6$ of $K_n$ is exceptional, yielding the Seifert fibered space $M \left(-1; \dfrac{1}{2} , \dfrac{2n+1}{4n+4} , \dfrac{2}{4n+5} \right)$. 

    For each of the other six short fillings $s\in \{-1,0,1,3,4,5\}$ we use SnapPy to verify that $L(*,s,*)$
     is hyperbolic. Since all other fillings of the remaining cusps yielding $K_n$ are sufficiently long, this implies that all other fillings of $K_n$ are hyperbolic.
\end{proof}

\subsection{Symmetry-exceptional slopes}

We continue by classifying the symmetry-exceptional slopes. Note that the symmetry exceptional slopes of $K_1$ were already classified as $\{12,15,16\}$ in~\cite{BKM}.

\begin{lemma}\label{lem:sym_exc}
    For each $n\geq 2$, the only symmetry-exceptional slopes of $K_n$ are the slopes $8n+4$ and $8n+8$, which both yield manifolds with symmetry group $(\mathbb Z_2)^2$.
\end{lemma}

\begin{proof}
    We use the same strategy as in the proof of Lemma~\ref{lem:exc}. But instead of the 6-theorem, we use the following. Let $M$ be a cusped hyperbolic manifold. Then, a recent result of Futer--Purcell--Schleimer~\cite{FPS} implies that there exists a computable bound $B(M)$ such that every filling of $M$ with slopes of normalized lengths longer than $B(M)$ has the same symmetry group as $M$, cf.~\cite{BKM2}.

    First, we use SnapPy to compute the bound $B(L)$, where $L$ is the link of the surgery description of the $K_n$. Then we check that the slopes $1/(2n+3)$ and $-1/(n+1)$ on the first and third cusps of $L$ have normalized length longer than $B(L)$ for all $n\geq 9$. Hence, any symmetry-exceptional slope of $K_n$ for $n\geq9$ must arise from a slope of normalized length less than $B(L)$ on the middle cusp of $L$. 
    
    The slopes on this cusp, whose normalized lengths are less than the Futer--Purcell--Schleimer bound $B(L)$, form a finite set. For all slopes $s$ in this set except
    \[
        \{\infty,0,2,4\},
    \]
    we compute that $L(*,s,*)$ has symmetry group isomorphic to $\mathbb{Z}_2$. Therefore, the only slopes that can give symmetry-exceptional fillings of $K_n$ are the images of $\{\infty,0,2,4\}$.

 The slope $\infty$ gives the meridional filling of $K_n$ and hence the manifold $S^3$, while the slope $2$ is exceptional by the previous lemma and therefore does not contribute a symmetry-exceptional hyperbolic filling.
It remains to consider the slopes $4$ and $0$. With SnapPy, we compute that the manifolds
    \[
         L(*,4,*)
        \quad\text{and}\quad
        L(*,0,*)
    \]
    both have symmetry group isomorphic to $(\mathbb{Z}_2)^2$. Since the normalized lengths of the slopes on the first and last cusp of $L$ yielding $K_n$ for $n\geq9$ are larger than $B(L)$, it follows that the images of $4$ and $0$ are symmetry-exceptional slopes for $K_n$ yielding manifolds with symmetry group $(\mathbb{Z}_2)^2$.

    Finally, for $n<9$, we build the knots $K_n$, compute the Futer--Purcell--Schleimer bound $B(K_n)$, and compute the symmetry groups of all fillings with length less than $B(K_n)$. This yields that also for $2\le n< 9$ the only symmetry exceptional slopes are $8n+4$ and $8n+8$. 
\end{proof}

\subsection{Thin slopes}

\begin{proof}[Proof of Theorem~\ref{thm:no_thin}]
    It was already shown in \cite{BKM} that $K_1$ has no thin slope. So in the following, we can assume that $n>1$. As explained at the beginning of this section, every thin slope of $K_n$ has to be exceptional or symmetry-exceptional. 
    
    By Lemma~\ref{lem:exc}, the only exceptional slope of $K_n$ is $8n+6$, yielding the small Seifert fibered space $M \left(-1; \dfrac{1}{2} , \dfrac{2n+1}{4n+4} , \dfrac{2}{4n+5} \right)$. In Lemma 2.2 of \cite{BK}, it is shown that this Seifert fibered space is an L-space. Thus, Lemma 4.1 from \cite{BKM} implies that we can write this Seifert fibered space uniquely as a double branched cover. This implies that the only possible branching set has to be the tangle filling $T_n(8n+6)$, which was already shown above to be thick.

    By Lemma~\ref{lem:sym_exc}, the only symmetry-exceptional slopes of $K_n$ are $8n+6\pm2$, which both yield manifolds with symmetry group $(\mathbb Z_2)^2$. Thus Lemma 2.4 of \cite{BKM2} shows that we can write these manifolds in exactly three ways as double branched cover over links in $3$-manifolds (possibly different from $S^3$). One of these is the tangle filling $T_n(8n+6\pm2)$, which is already known to be thick.

However, the others have quotient manifolds different from $S^3$ as shown in Lemmas \ref{lem:8n+4} and \ref{lem:8n+8} below.
Thus, we confirmed that $K_n(8n+6\pm 2)$ has the unique description as the double branched cover of $S^3$.  
\end{proof}

\begin{lemma}\label{lem:8n+4}
$K_n(8n+4)$ has the unique description as the double branched cover of $S^3$ over a link.
\end{lemma}

\begin{proof}
For the surgery diagram as shown in Figure \ref{fig:8n+4},
it is straightforward to verify that
the quotient of this diagram under the involution around the horizontal axis $A_h$ shown there yields $T_n(8n+4)$ as its branching set.
Hence, this surgery diagram gives the manifold $K_n(8n+4)$.

\begin{figure}[htbp]
\includegraphics*[width=0.6\textwidth]{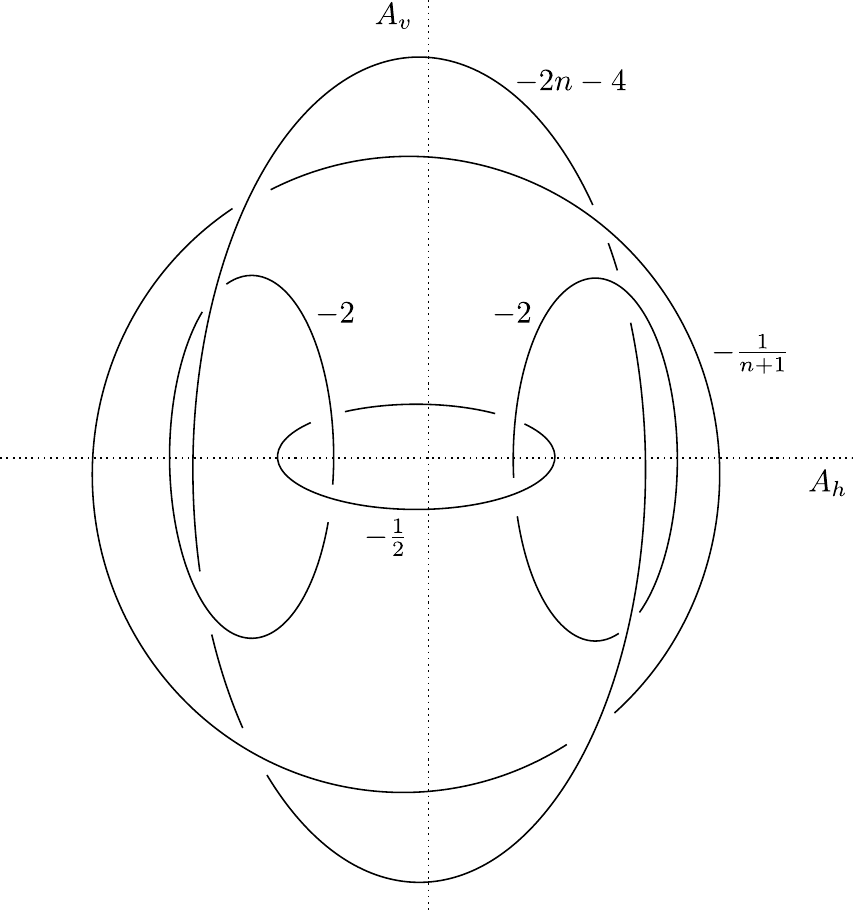}
\caption{The surgery diagram yields $K_n(8n+4)$.
}\label{fig:8n+4}
\end{figure}

On the other hand, this diagram has other symmetries.
If we take the quotient under the involution around the vertical axis $A_v$, then
the ambient manifold is seen to be the lens space $-L(2,1)$.
See Figure \ref{fig:8n+4another}(a), where we ignore the branching set, given by the surgery curves that intersect the symmetry axis.

Another symmetry is the involution around the axis perpendicular to the projection plane.
Then Figure \ref{fig:8n+4another}(b) shows that the ambient manifold after taking the quotient is the lens space $-L(2n+1,n+1)$.
\end{proof}

\begin{figure}[htbp]
\includegraphics*[width=0.8\textwidth]{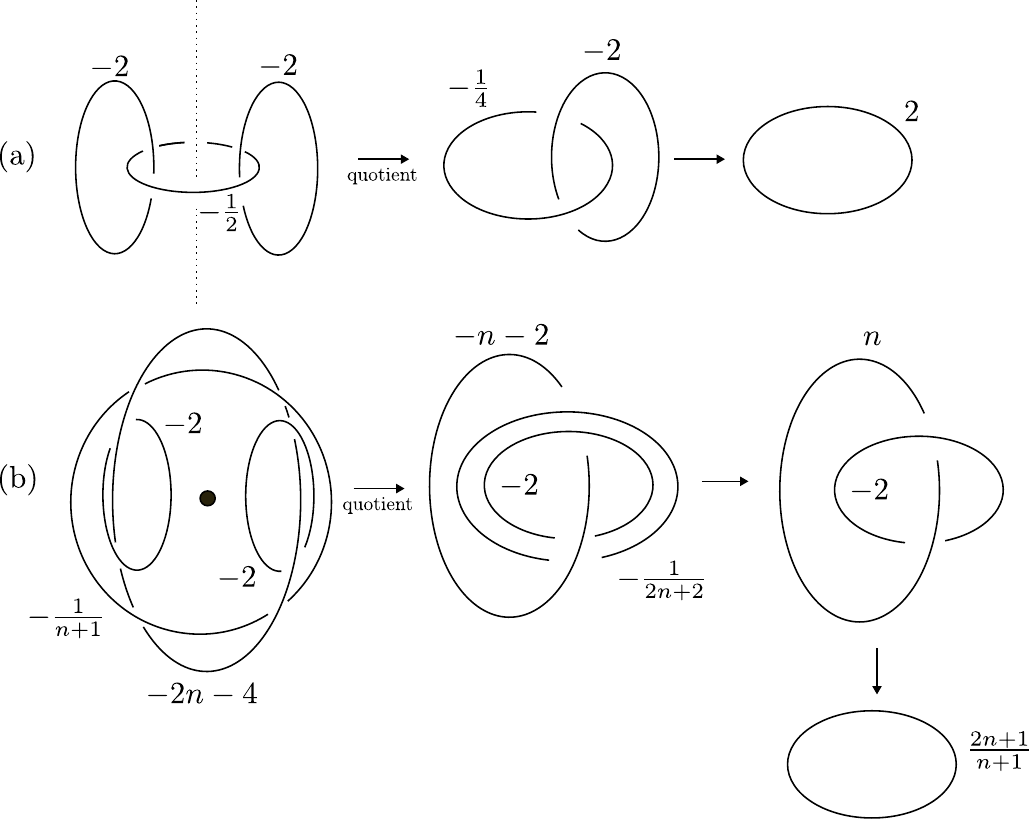}
\caption{(a) shows that the quotient manifold for the involution
around the vertical axis is $-L(2,1)=\mathbb R P^3$.
(b) shows that the quotient manifold for the involution around
the perpendicular axis is the lens space $-L(2n+1,n+1)$.
}\label{fig:8n+4another}
\end{figure}

\begin{lemma}\label{lem:8n+8}
$K_n(8n+8)$ has the unique description as the double branched cover of $S^3$ over a link.
\end{lemma}

\begin{proof}
The argument is similar to that of Lemma \ref{lem:8n+4}.
Figure \ref{fig:Tn(8n+8)} shows the surgery diagram yielding $K_n(8n+8)$.
This is confirmed by the fact that taking the quotient under the involution around the axis $A_h$ shown there yields the branching set  $T_n(8n+8)$.

\begin{figure}[htbp]
\includegraphics*[width=0.6\textwidth]{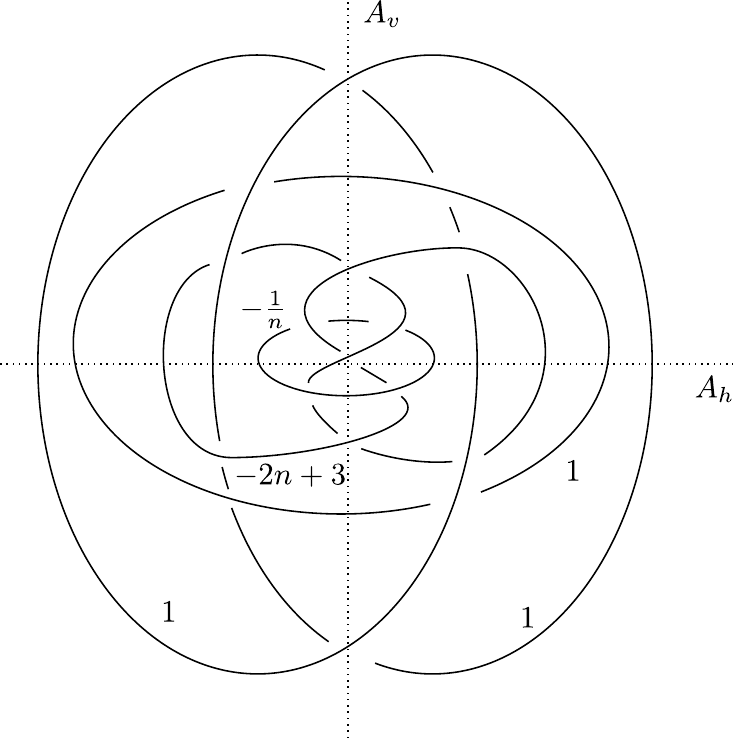}
\caption{This surgery diagram yields $K_n(8n+8)$.
}\label{fig:Tn(8n+8)}
\end{figure}

As before, this surgery diagram has other symmetries.
If we take the quotient around the axis perpendicular to the
projection plane, the ambient manifold is $L(2,1)$, as shown in Figure \ref{fig:Tn(8n+8)other}(a).
Finally, taking the quotient around the vertical axis $A_v$ yields $L(4n+4,2n+1)$ as shown in Figure \ref{fig:Tn(8n+8)other}(b).
\end{proof}

\begin{figure}[htbp]
\includegraphics*[width=0.8\textwidth]{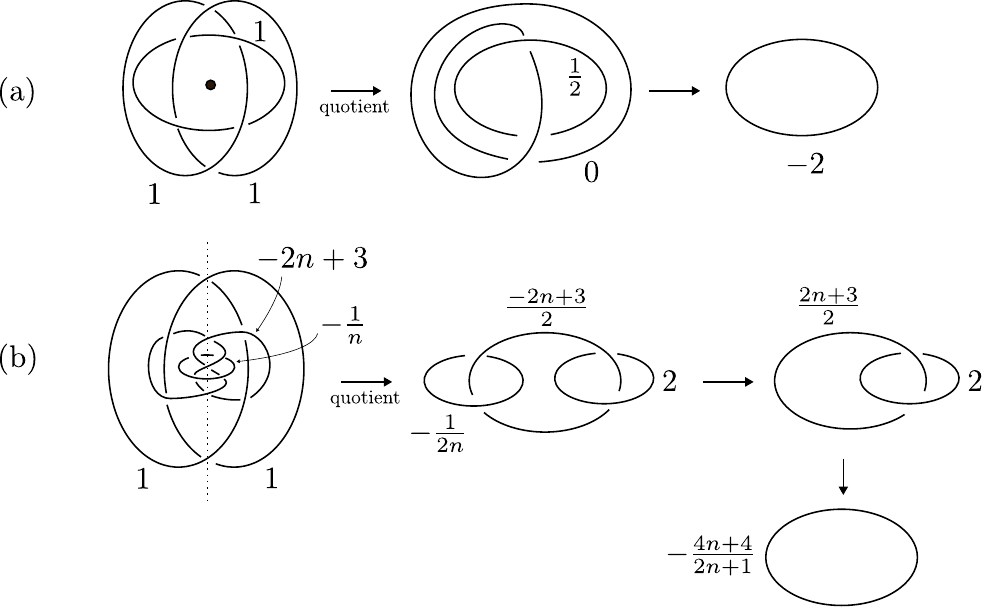}
\caption{(a) The quotient manifold under the involution around the perpendicular axis is $L(2,1)$.
(b) The quotient manifold under the involution around the vertical axis $A_v$ is $L(4n+4,2n+1)$.
}\label{fig:Tn(8n+8)other}
\end{figure}

\bibliographystyle{amsplain}

\end{document}